\documentclass[reqno,12pt,twosided,a4paper]{amsart}
\usepackage{amssymb,amsfonts,latexsym,mathrsfs,amsmath}
\usepackage[all]{xy}
\usepackage{array}

\usepackage{amstext}
\usepackage{amssymb}
\usepackage{amsfonts}
\usepackage{amsthm}
\usepackage{amsopn}
\usepackage{amsbsy}
\usepackage{layout}

\newtheorem{theorem}{Theorem}[section]
\newtheorem{lemma}[theorem]{Lemma}
\newtheorem{proposition}[theorem]{Proposition}
\newtheorem{corollary}[theorem]{Corollary}

\theoremstyle{definition}
\newtheorem{definition}[theorem]{Definition}

\newtheorem{remark}[theorem]{Remark}

\newcommand{\Z}{\mathbb{Z}}
\newcommand{\Q}{\mathbb{Q}}

\newcommand{\C}{\mathcal{C}}

\newcommand{\Ow}{\mathcal{O}}

\newcommand{\ot}{\otimes}

\newcommand{\ti}{\widetilde}

\newcommand{\Ga}{\Gamma}
\newcommand{\ra}{\rightarrow}
\newcommand{\ep}{\epsilon}
\renewcommand{\l}{\ell}
\newcommand{\la}{\lambda}
\newcommand{\wb}{\overline}

\newcommand{\dott}{\bullet}
\newcommand{\ga}{\gamma}

\begin{document}
\title[Hopf algebras via invariant theory]{Semisimple Hopf algebras via geometric invariant theory}
\author{Ehud Meir}
\address{University of Hamburg, Department of Mathematics, Bundesstrasse 55, 20146 Hamburg, Germany}
\email{meirehud@gmail.com}
\begin{abstract}
We study Hopf algebras via tools from geometric invariant theory.
We show that all the invariants we get can be constructed using the integrals of the Hopf algebra
and its dual together with the multiplication and the comultiplication,
and that these invariants determine the isomorphism class of the Hopf algebra.
We then define certain canonical subspaces $Inv^{i,j}$  of tensor powers of $H$ and $H^*$, 
and use the invariant theory to prove that these subspaces satisfy a certain non-degeneracy condition.
Using this non-degeneracy condition together with results on symmetric monoidal categories, we prove that the spaces $Inv^{i,j}$ can also be described as $(H^{\ot i}\ot (H^*)^{\ot j})^A$,
where $A$ is the group of Hopf automorphisms of $H$.
As a result we prove that the number of possible Hopf orders of 
any semisimple Hopf algebra over a given number ring is finite.
we give some examples of these invariants arising from the theory
of Frobenius-Schur Indicators, and from Reshetikhin-Turaev invariants of three manifolds.
We give a complete description of the invariants for a group algebra, proving that they all encode the number of homomorphisms from some finitely presented group to the group.
We also show that if all the invariants are algebraic integers, then the Hopf algebra satisfies Kaplansky's sixth conjecture:  the dimensions of the irreducible representations of $H$ divide the dimension of $H$.
\end{abstract}
\maketitle
\begin{section}{Introduction}\label{introduction}
In this paper we develop an approach for studying finite dimensional semisimple Hopf algebra over an algebraically closed field of characteristic zero $K$ by means of geometric invariant theory.
Two basic examples to keep in mind for such Hopf algebras are the group algebra $KG$ for a finite group $G$ and its dual, the function algebra $K[G]$.
A more evolved example is given as follows: if a finite group $G$ acts on a finite group $N$, then one can construct the semi-direct product
$KG\ltimes K[N]$, which is also a Hopf algebra. All these Hopf algebras are examples of \textit{group-theoretical Hopf algebras}.
In \cite{ENO} Etingof Nikshych and Ostrik defined the notion of a group theoretical fusion category, 
and defined group-theoretical Hopf algebras to be Hopf algebras whose representation categories are group theoretical.
They also asked whether every finite dimensional semisimple Hopf algebra is group theoretical.
In \cite{Nikshych}, Nikshych gave a counterexample, by presenting a family of finite dimensional semisimple Hopf algebras 
whose representation categories are only weakly group theoretical.
Nevertheless, all the known examples of finite dimensional semisimple Hopf algebras are constructed in one way or another from 
some group theoretical data. 

A lot is known about Hopf algebras of some restricted dimensions. 
To name a few examples, Zhu proved in \cite{Zhu} that if the dimension $p$ of a Hopf algebra $H$ is prime, 
then this Hopf algebra is isomorphic with the group algebra of the cyclic group of order $p$.
In case the dimension of $H$ is $pq$ where $p$ and $q$ are two distinct prime numbers
it is known that $H$ is isomorphic either with a group algebra or a dual group algebra (see \cite{EG1}, \cite{Som1} and \cite{GW}).
For more classification result see the work \cite{Natale1} of Natale. 

In general, classifying \textit{all} finite dimensional semisimple Hopf algebras in terms of some group-theoretical data is hard, and seems to be out of reach at the moment.
The source of difficulty can be understood in the following way: we do know that if $H$ is semisimple then $H^*$ is also semisimple,
by Larson-Radford Theorem. We can then write both $H$ and $H^*$ as the direct sum of matrix algebras over $K$.
The problem is that even though we know the algebra structure and the coalgebra structure of $H$, 
we do not know how these structures interact. It is possible that two non-isomorphic Hopf algebras will have the same algebra and coalgebra structures (for example $K\Z/4$ and $K\Z/2\times \Z/2$ or $K Q_8$ and $K D_8$). 
The goal of this paper is to present tools from geometric invariant theory in order to study this problem.
Our starting point will be to ``translate'' the classification question into an algebraic-geometric question.
In Section \ref{constructingX} we will construct a variety $X$ and an algebraic group $G$ which acts on $X$,
such that the orbits of this action correspond to the different isomorphism types of Hopf algebras with given algebra and coalgebra structures.
We will then show that $G$ has only finitely many orbits in $X$ and that they are all of the same dimension and therefore closed.
This will follow from a theorem of Stefan (see \cite{Stefan}) and a theorem of Radford (see \cite{Radford-aut}) 
about the finiteness of the group of automorphisms of a finite dimensional semisimple Hopf algebra.
This fact enables us to apply the techniques of Mumford's Geometric Invariant Theory (or GIT).
The main result from GIT which we shall use is the fact that 
the invariant polynomials in $K[X]^G$ define 
the isomorphism type of the Hopf algebra (In fact we will have an isomorphism $K[X]^G\cong K[X/G]$).

In Sections \ref{tiG-invariants} and \ref{G-invariants} we will describe the invariants explicitly, using the methods of \cite{Procesi}.
More generally, we will define some canonical elements in $H^{i,j}:=H^{\ot i}\ot (H^*)^{\ot j}$ in the following way: 
let $\l\in H$ and $\la\in H^*$ be the integrals in $H$ and in $H^*$ which satisfy $\ep(\l) = \la(1) = dim(H)$
(these are also the characters of the regular representations of $H^*$ and $H$ respectively).
Consider the element $\l^{\ot a}\ot \la^{\ot b}\in H^{a,b}$. For some $a$ and $b$.
By applying comultiplication repeatedly to some of the tensor factors
we can get an element in $H^{m+i,m+j}$ for some $m,i$ and $j$. By applying some permutations in $S_{m+i}$ and $S_{m+j}$ 
and by pairing the first $m$ tensor factors of $H$ with the first $m$ tensor factors of $H^*$ we get an element in $H^{i,j}$.
We call such elements \textit{$(i,j)$-basic invariants}. In particular, $(0,0)$-basic invariants are scalars (to abbreviate, we will just call them basic invariants). 
In Section \ref{G-invariants} we will prove the following theorem:
\begin{theorem}\label{main1}
The basic invariants span $K[X]^G\cong K[X/G]$
after a finite localization. 
They therefore determine the isomorphism type of $H$.
\end{theorem}
\begin{remark}
This theorem was originally proved by Datt, Kodiyalam and Sunder (see Theorem 11 in \cite{dks}).
Their proof also relies on geometric invariant theory, but the variety $X$ in their construction
contains all finite dimensional semisimple Hopf algebras of a given dimension,
not only those with a specific algebra and coalgebra structure.
As a result, the group $G$ is also different.
I include my alternative construction and proof here instead of just referring to \cite{dks}
because it will be used in the rest of the paper.
In addition, the group $G$ which appears here 
contains a finite index subgroup $\ti{G}$ whose invariants (which we shall call here the character basic invariants) are interesting in their own right.
\end{remark}  
	
One useful consequence of Theorem \ref{main1} is that it gives us a uniform description of the invariants, which do not depend on the dimension of the Hopf algebra,
or the dimensions of its irreducible representations and co-representations (even though the variety $X$ and the group $G$ do depend on them).
We denote by $K_0$ the subfield of $K$ which is generated by the basic invariants,
and by $Inv^{i,j}$ the $K$ subspace of $H^{i,j}$ which is generated by the $(i,j)$-basic invariants.
An immediate result of Theorem \ref{main1} is the following:
The fact that the orbit of $H$ in $X$ is an open (and closed) subset implies that $H$ can already be defined over $\wb{\Q}$, the algebraic closure of $\Q$
(this follows from the fact that $X$ and $G$ are already defined over $\wb{\Q}$).
The Galois group of $\wb{\Q}/\Q$ acts on the set of all isomorphism types of Hopf algebras defined over $\wb{\Q}$.
Since $H$ is determined by its invariants, we have the following theorem:
\begin{theorem}\label{main2}
The field $K_0$ is contained in $\wb{\Q}$, and we have an equality $K_0 = \wb{\Q}^{stab([H])}$. 
\end{theorem}

We have a canonical pairing $Inv^{i,j}\ot_{K} Inv^{j,i}\ra K$ arising from the pairing $H^{i,j}\ot_K H^{j,i}\ra K$.
In Section \ref{invsspaces} we will apply Theorem \ref{main1} and prove the following theorem:
\begin{theorem}\label{main3}
The pairing $Inv^{i,j}\ot_{K}Inv^{j,i}\ra K$ is non-degenerate. 
\end{theorem}
We will then construct, in Section \ref{Tannaka}, a symmetric monoidal category out of these spaces.
By applying Tannaka Reconstruction Theorem we will prove the following result:
\begin{theorem}\label{main4}
The subspace $Inv^{i,j}$ of $H^{i,j}$ is equal to $(H^{i,j})^{Aut_{Hopf}(H)}$.
\end{theorem}
It is quite easy to prove inclusion in one direction, namely that $Inv^{i,j}\subseteq (H^{i,j})^{Aut_{Hopf}(H)}$.
The inclusion in the other direction is less clear, and its proof relies heavily on the rigid nature of finite dimensional semisimple Hopf algebras.
Notice also that Theorem \ref{main3} follows easily from Theorem \ref{main4} and Maschke's Theorem.
Proving that Theorem \ref{main3} implies Theorem \ref{main4} will be harder, 
and will require us to use results on Symmetric monoidal categories.

As a generalization of the basic $(i,j)$-invariants we construct $(i,j)$-character basic invariants.
The $(i,j)$-character basic invariants are constructed in the same way as the
$(i,j)$-basic invariants, with the difference that instead of using a tensor product of $\l$ and $\la$,
we are allowed to use tensor products of arbitrary characters of $H^*$ and of $H$.
We denote by $\ti{Inv}^{i,j}\subseteq H^{i,j}$ the subspace spanned by $(i,j)$-character basic invariants.
Theorems \ref{main3} and \ref{main4} can be easily generalized in the following way:
\begin{theorem}\label{main4.5}
The pairing $\ti{Inv^{i,j}}\ot \ti{Inv^{j,i}}\ra K$ is non-degenerate,
and we have $\ti{Inv^{i,j}}= (H^{i,j})^{Aut_{Hopf}^0(H)}$
where $Aut_{Hopf}^0(H)\subseteq Aut_{Hopf}(H)$ is the subgroup of all Hopf automorphisms which fix all the characters 
and cocharacters of $H$.
\end{theorem}
The $(0,0)$-character basic invariants form a set of scalars.
Since there are more character basic invariants then basic invariants, they will determine a more specific structure.
In Section \ref{tiG-invariants} we will prove the following theorem:
\begin{theorem}\label{tildeGinv} The $(0,0)$-character basic invariants determine the isomorphism type of the ordered tuple $(H,W_1,\ldots, W_c,V_1,\ldots,V_d)$
 where $H$ is a finite dimensional semisimple Hopf algebra, $W_i$ are the irreducible representations of $H$ and $V_j$ are the irreducible representation of $H^*$. 
\end{theorem}
The $(1,0)$ and $(0,1)$-character basic invariants appeared in \cite{CM1} and \cite{CM2}
in the study of orders of Hopf algebras by means of their character theory.
We have the following corollary of Theorem \ref{main4.5}, which we shall prove in Section \ref{finiteorders}
\begin{theorem}\label{main5}
Let $L\subseteq K$ be a number field, and let $H$ be a semisimple finite dimensional Hopf algebra over $L$.
Then $H$ has at most finitely many Hopf orders over $\Ow_L$.
\end{theorem}
This finiteness result is relatively easy to prove by the methods of \cite{CM1} and \cite{CM2} in case the Hopf algebra $H$ is a group algebra.
It follows from the fact that in this case $Aut^0_{Hopf}(H)=1$. 
Here we prove that it is in fact true for any finite dimensional semisimple Hopf algebra.

Finally, in Sections \ref{ex1} and \ref{ex2} we will give some concrete examples of these invariants.
We will show that in case $H=KG$ is a group algebra (where $G$ is any finite group),
then all the basic invariants are of the form $|G|^{a}\# Hom_{Grp}(P,G)$ for some finitely generated group $P$,
where $a$ is the number of relations in some finite presentations of $P$ (see also \cite{dks}).
We will also show that for a general Hopf algebra some of the specific basic invariants are well known,
for example the Frobenius-Schur indicators and the Reshetikhin-Turaev invariants of three dimensional manifolds.
By studying some specific invariants we will prove in Section \ref{ex2} the following result, 
which relates the invariants to Kaplansky's sixth conjecture:
\begin{theorem}\label{main6}
If all the basic invariants of $H$ are algebraic integers, then $H$ satisfies Kaplansky's Sixth Conjecture: 
the dimensions of every irreducible representation of $H$ and of $H^*$ divide the dimension of $H$.
\end{theorem}
\end{section}


\begin{section}{preliminaries}\label{prelim}
\begin{subsection}{Hopf algebras}
A \textit{bialgebra} $H$ over a field $K$ 
is an algebra with unit $(H,m,1)$ which is also a coalgebra with a counit $(H,\Delta,\ep)$
such that the counit $\ep$ and the comultiplication $\Delta$ are algebra maps
(or equivalently, such that $m$ and $1:K\ra H$ are coalgebra maps).
This means that $\ep(1) = 1$, $\Delta(1) = 1\ot 1$, $\ep(xy) = \ep(x)\ep(y)$
and that \begin{equation}\label{compatib}\Delta(xy) = \Delta(x)\Delta(y)\end{equation}
for $x,y\in H$.
The vector space $Hom_K(H,H)$ becomes then an algebra by the convolution product:
$$(f\star g)(x) = f(x_1)g(x_2)$$
where we use the Sweedler notation $\Delta(x) = x_1\ot x_2$.
A bialgebra is called a \textit{Hopf algebra} if the identity map has a two sided inverse in
$Hom_K(H,H)$ with respect to the convolution product. 
This inverse (if it exists) is called the \textit{antipode} of $H$
and is denoted by $S$.

If $H$ is a finite dimensional Hopf algebra then $H^*$
is again a Hopf algebra, with multiplication $\Delta^*$ and comultiplication $m^*$. 
The natural isomorphism $Hom_K(H,H)\cong H\ot H^*$ of vector spaces 
is an isomorphism of algebras where the left hand side is an algebra with respect to the convolution product, 
and the right hand side is an algebra with respect to the tensor product of the two algebras $H$ and $H^*$.

A lot is known about the structure of a finite dimensional Hopf algebra $H$ if it is also known to be semisimple as an algebra (we shall assume that this is the case for the rest of this paper.
Everything that will not be proved here can be found in \cite{LR1}, \cite{LR2}, \cite{Montgomery} and \cite{EG2}). 
Indeed, by Larson-Radford Theorem we know that $H^*$ is also semisimple.
We also know that $S^2=Id$ and by a result of Etingof and Gelaki we know that the exponent of the Hopf algebra is finite.
This means that for some natural number $m$ it holds that $x_1x_2\cdots x_m = \ep(x)$ for every $x\in H$.
In particular, the identity $Id_H\in H\ot H^*\cong Hom_K(H,H)$ has a finite order in the convolution algebra,
and therefore $S = Id^{\star m-1}$ so that \begin{equation}\label{antipode}S(x) = x_1x_2\cdots x_{m-1}.\end{equation}

Since $H$ is semisimple and $K$ is algebraically closed we have by the Wedderburn Theorem an isomorphism $H\cong \oplus_i End(W_i)$ of algebras, 
where $\{W_i\}$ are the distinct types of irreducible representations of $H$.
In a similar way we have an isomorphism $H^*\cong \oplus_j End(V_j)$ where $\{V_j\}$ 
are the distinct types of irreducible representations of $H^*$. 
We will recall here some useful identities that the integrals in $H$ and in $H^*$ satisfy.
We first recall that a left integral $\l\in H$ is an element which satisfies $x\l = \ep(x)\l$ for every $x\in H$
(a left integral $\la\in H^*$ and right integrals in $H$ and in $H^*$ are defined in the obvious way).
It is known that any finite dimensional Hopf algebra contains a one dimensional subspace of left integrals.
In case $H$ is semisimple, this subspace will be spanned by $e_1$, the central idempotent which corresponds 
to the trivial representation of $H$, and left and right integrals coincide.
Let then $\l\in H$ and $\la\in H^*$ be integrals which satisfy $\ep(\l)=\la(1) = dim(H)$
(it is known that in a semisimple Hopf algebra the counit does not vanish on non-trivial left integrals). 
It is known that $\l$ is also the character of the regular representation of $H^*$  and similarly
$\la$ is the character of the regular representation of $H$.
Since $H$ is semisimple, $\l=dim(H)e_1$.
Thus, if $\psi$ is an irreducible character of $H$ we have $\psi(\l) = \delta_{\psi,\ep}dim(H)$.
Recall that since $\l$ is an integral it holds that for every $x\in H$ we have 
\begin{equation}\label{intcoprod1} x\l_1\ot \l_2 = \l_1\ot S(x)\l_2\textrm{ and } \l_1x\ot \l_2 = \l_1\ot \l_2S(x)\end{equation}
It then follows that $$x\l_1\ot S(\l_2) = \l_1\ot S(\l_2)x.$$
From this we can prove that if we write $\{e^i_{j,k}\}$ for the matrix units in $End(W_i)$ with respect to some basis, then it holds that
\begin{equation}\label{integral1} \l_1\ot S(\l_2) = \sum_i\frac{dim(H)}{dim(W_i)}\sum_{j,k}e^i_{j,k}\ot e^i_{k,j}.\end{equation}
Moreover, the map $P:H\ra H$ given by $$P(x) = \frac{1}{dim(H)}\l_1xS(\l_2)$$ 
is a projection of $H$ onto the center of $H$, and it can be written explicitly as
$P(M) = \frac{1}{dim(W_i)}tr(M)e_i$ for $M\in End(W_i)$, where we denote by $e_i$ the identity element of $End_K(W_i)$.
Another result of Equation \ref{integral1} and the fact that $\la$ is the character of the regular representation of $H$,
is that \begin{equation}\label{id2} dim(H)Id_H= \la_1(\l_1)S(l_2)\ot \la_2 \in H\ot H^*\cong End_K(H)\end{equation} and since $S^2=Id$ we have 
\begin{equation}\label{id1}dim(H)S = \la_1(\l_1)l_2\ot \la_2.\end{equation}
We will use this equation later, in order to prove certain properties of the spaces $Inv^{i,j}$ (which will be defined immediately).
Integrals induce a linear isomorphism between $H$ and $H^*$.
Indeed, the maps $\rho:H\ra H^*$ $x\mapsto \la_1(x)\la_2$ and $\mu:H^*\ra H$ $f\mapsto f(\l_1)\l_2$ are both linear isomorphisms.
Their composition can be seen to equal to $\mu\rho(x) = dim(H)S(x)$ (this is essentially Equation \ref{id1}).
Moreover, these maps give us a nice relation between the characters of $H$ and the center of $H$:
if we denote by $\psi_i$ the character of $W_i$, then a direct calculation shows that \begin{equation}\label{idem-char}\mu(S(\psi_i)) = \frac{dim(H)}{dim(W_i)}e_i 
\textrm{ and } \rho(e_i) = dim(W_i)\psi_i.\end{equation}

We now introduce a family of subspaces $Inv^{i,j}$ of $H^{i,j}:=H^{\ot i}\ot (H^*)^{\ot j}$, where $i$ and $j$ are two natural numbers.
We have a natural pairing $ev:H\ot H^*\ra K$. We denote by $ev:H^{i,j}\ra H^{i-1,j-1}$ also the evaluation map on the first copy of $H$ with the first copy of $H^*$.
\begin{definition}\label{invs}
An $(i,j)$-character basic invariant is an element of $H^{i,j}$ of the form $T_1T_2\cdots T_l(\mu_1\ot \cdots\ot\mu_a\ot \nu_1\ot\cdots\ot\nu_b)$,
where $\mu_s\in H^*$ are characters of representations of $H$, 
$\nu_t\in H$ are characters of representations of $H^*$,  
and the linear maps $T_i$ are either comultiplication on $H$, comultiplication on $H^*$, a permutation of the tensor factors of $H$, a permutation of the tensor factors of $H^*$ or the evaluation map. 
We denote by $\ti{Inv}^{i,j}$ the space spanned by all the $(i,j)$- character basic invariants.
In case all the characters $\mu_s$ and $\nu_t$ are the characters of the regular representations of $H$ and $H^*$ respectively, we call the resulting element an $(i,j)$-basic invariant. 
We denote by $Inv^{i,j}$ the space spanned by all the $(i,j)$-basic invariants. 
\end{definition}
So for example $\l\in Inv^{1,0}$ and $\l_1\ot \l_2\ot \la\in Inv^{2,1}$. 
Alternatively, one can define $(Inv^{i,j})$ (or $\ti{Inv}^{i,j}$) as the smallest collection of subspaces of $H^{i,j}$ such that:\\
1. It holds that $\l\in Inv^{1,0}$ and $\la\in Inv^{0,1}$ (all characters of $H$ are contained in $\ti{Inv}^{0,1}$ and all characters of $H^*$ are contained in $\ti{Inv}^{1,0}$). \\
2. The collection $Inv^{i,j}$ ($\ti{Inv}^{i,j}$) is closed under comultiplication in $H$ and in $H^*$.\\
3. The collection $Inv^{i,j}$ ($\ti{Inv}^{i,j}$) is closed under the action of the symmetric groups on tensor powers.\\
4. The collection $Inv^{i,j}$ ($\ti{Inv}^{i,j}$) is closed under tensor product (in the sense that $Inv^{i,j}\ot Inv^{a,b}\subseteq Inv^{i+a,j+b}$ under the identification $H^{i,j}\ot H^{a,b}\cong H^{i+a,j+b}$).\\
5. The collection $Inv^{i,j}$ ($\ti{Inv}^{i,j}$) is closed under the evaluation map.\\
Notice in particular that for $(i,j)=(0,0)$ we get two collections of scalars, namely $Inv^{0,0}$ and $\ti{Inv}^{0,0}$.
These collection will appear again later as the values of certain invariant polynomial functions.
We will next prove that the collection $Inv^{i,j}$ has some closure properties, and contains certain elements.
The same results hold for the collection $(\ti{Inv}^{i,j})$ and the proofs are similar, in case they do not follows directly from the fact that $Inv^{i,j}\subseteq \ti{Inv}^{i,j}$.
\begin{lemma} 
The subspace $Inv^{i,j}$ contains $S$ and $Id_H\in End_K(H,H)\cong H^{1,1}$ and is closed under multiplication in $H$ and in $H^*$. 
\end{lemma}
\begin{proof}
The first claim follows directly from Equations \ref{id2} and \ref{id1}. 
The second claim follows from the fact that the multiplication in $H$ can be written (up to a nonzero scalar) as
$$\l^2_2\ot \la^1_1\ot \la^1_2\la^1_3(\l^1_1)\la^2_1(\l^1_2)\la^2_2(l^2_1)\in Inv^{1,2}$$
where $\l^1$ and $\l^2$ are two copies of $\l$ (and similarly for $\la$).
The proof is just a direct verification, using repeatedly Equations \ref{intcoprod1} and \ref{id1}. 
Since the collection of subspaces $Inv^{i,j}$ is closed under tensor product and evaluation, this proves that it is closed under the operation of multiplication.
\end{proof}
We will need some more specific Hopf algebra identities later.
We begin with the following identity, which first appeared in \cite{CW}:
\begin{lemma}
We have an equality \begin{equation}\label{integral2}\l^1_1\ot \l^2_1S(\l^1_2)S(\l^2_2) = \sum_i \frac{dim(H)^2}{dim(W_i)^2}e_i\ot e_i\end{equation}
where $\l^1$ and $\l^2$ are two copies of $\l$.
\end{lemma}
\begin{proof}
This follows from Equation \ref{integral1} and what we have proved about the map $P$ above.
\end{proof}
We write $c_H:=\l^1_1\ot \l^2_1S(\l^1_2)\l^2_2\in H\ot H$.
It thus holds that $c_H\in Inv^{2,0}$.
We can use this element to distinguish characters of different dimensions.
More precisely, we claim the following:
\begin{lemma}\label{mess1}
Let $d$ be a natural number.
The element $$c^{2,d}_H=\sum_{i, dim(W_i)=d}e_i\ot e_i$$ can be written as a polynomial in ${c_H}$.
As a result, for every $d$ and $n$ the element $c^{n,d}_H=\sum_{i,dim(W_i)=d}e_i^{\ot n}$
belongs to $Inv^{n,0}$.
\end{lemma}
\begin{proof}
The first part of the lemma follows from the fact that there exists a polynomial $f$ with coefficients in $\Q$
such that $f(\frac{dim(H)^2}{d^2})=1$ and $f(\frac{dim(H)^2}{d'^2})=0$ for every $d'\neq d$ which is 
a dimension of an irreducible representation of $H$.
The second part of the lemma follows from the fact that we have
$${c^{n,d}_H} = (c^{2,d}_H)_{1,2}(c^{2,d}_H)_{2,3}\cdots (c^{2,d}_H)_{n-1,n}$$
where $(c^{2,d}_H)_{1,2} = {c^{2,d}_H}\ot {1_H}^{\ot (n-2)}$ and similarly for the other indices.
Notice that for $n=1$ the result holds because $m({c^{2,d}_H})={c^{1,d}_H}$ where we denote by 
$m$ the multiplication in $H$.
\end{proof}
The following corollary will be used in the proof of the main result of Section \ref{G-invariants}
\begin{corollary}\label{ekses}
For every sequence of natural numbers $a_1,\ldots a_n$ the expression
\begin{equation}\label{mess2}\sum_{\substack{(i_1,\ldots,i_n):\\ dim(W_{i_j})=d\, , |\{i_1,\ldots, i_n\}|=n}} e_{i_1}^{\ot a_1}\ot e_{i_2}^{\ot a_2}\ot\cdots\ot e_{i_n}^{\ot a_n}\end{equation}
belongs to $Inv^{a_1+a_2+\cdots + a_n,0}$.
\end{corollary}
\begin{proof}
By multiplying with ${c^{2,d}_H}$ (which belongs to $Inv^{2,0}$) enough times we can reduce to the case where $a_i=1$ for every $i$.
Notice that if $n$ is bigger than the number of non-isomorphic irreducible representations
of dimension $d$ then the sum is zero, and the result holds trivially.

We will prove the result by induction on $n$.
The case where $n=1$ follows from Lemma \ref{mess1}. 
If the result holds for $n-1$ then the element
$$x_{n-1}= \sum_{\substack{(i_1,\ldots,i_{n-1})\\ dim(W_{i_j})=d\, , |\{i_1,\ldots, i_{n-1}\}|=n-1}} e_{i_1}\ot e_{i_2}\ot\cdots\ot e_{i_{n-1}}$$
belongs to $Inv^{n-1,0}$.
The same then holds for $y:=x_{n-1}\ot {c_H}^{1,d}$.
By multiplying the idempotent $1_{H\ot  H}-{c_H}^{2,d}$ with some of the pairs of tensors in $y$, we get the result,
since the $Inv^{i,j}$ spaces are closed under multiplication.
\end{proof}
For future reference, We shall denote the expression in Equation \ref{mess2} by $x^{a_1,\ldots, a_n}_{d,n}$
\end{subsection}

\begin{subsection}{Geometric Invariant theory}
Let $X$ be an affine variety, and let $G$ be a reductive algebraic group which acts on $X$ rationally.
In this paper, $X$ will be the variety of all Hopf algebra structures with given algebra and coalgebra structures,
and $G$ will be (virtually) a product of $PGL_n$'s (The variety $X$ and the group $G$ will be constructed in Section \ref{constructingX}).
The following theorem is a collection of results from Geometric Invariant Theory which we will use in this paper.
\begin{theorem}\label{GIT} (see Chapter 3 of \cite{Newstead})
Assume that $G$ acts on $X$ with finite stabilizers. Then the orbit space $X/G$ is also an affine variety.
Moreover, we have an isomorphism $K[X/G]\cong K[X]^G$, and the natural map $X\ra X/G$ corresponds to the inclusion of algebras $K[X]^G\ra K[X]$.
We have a one to one correspondence between closed $G$-stable subsets of $X$ and closed subsets of $X/G$. 
Therefore, if $I\subseteq K[X]$ is a radical $G$-stable ideal of $X$, then $I\neq 0$ if and only if $I^G\neq 0$.
\end{theorem}
Notice that the fact that all the stabilizers of $G$ in $X$ are finite implies that all the orbits have the same dimension.
Therefore, all the orbits are closed. 
The next thing that we need to do is to describe the invariants of some actions of some specific algebraic groups.
Our group $G$ will be a finite extension of a product of projective general linear groups. 
For the finite group part, we have the next lemma, which follows easily from Maschke's Theorem:
\begin{lemma}\label{invfinite}
If a finite group $G$ acts on a $K$-algebra $C$, then the map $c\mapsto \frac{1}{|G|}\sum_{g\in G} g(c)$
is a projection of $C$ onto $C^G$
\end{lemma}
Notice that the projection from the lemma is not necessarily a projection of algebras.
Next, we deal with the action of $PGL_n$. 
We will follow closely the work of Procesi (see \cite{Procesi}). We begin with the following lemma, whose proof is straightforward:
\begin{lemma}\label{invpgl}
Let $V$ be a finite dimensional rational representation of an algebraic group $G$ over $K$.
Then $G$ acts on the homogeneous affine algebra $K[V]\cong S^{\dott}(V^*)$.
The space $S^nV$ is a direct summand of $V^{\ot n}$, and therefore $K[V]^n=Hom_K(S^nV,K)$ is a direct summand of $Hom_K(V^{\ot n}\ot K)$.
The projection $Hom_K(V^{\ot n},K)\ra Hom_K(S^nV,K)$ sends $f:V^{\ot n}\ra K$ to the polynomial function $\ti{f}(v) = f(v\ot v\ot\cdots\ot v)$.
Moreover, this projection restricts to the $G$-invariant part, and so we have a surjective map $Hom_G(V^{\ot n},K)\ra Hom_G(S^nV,K)= (K[V]^n)^G$.
\end{lemma}
We thus see that a description of the invariants of $Hom_K(V^{\ot n},K)$ for all $n$ will give us the invariants in $K[V]$.
The next theorem is based on the Schur-Weyl Duality. 
It was originally proved by Procesi in order to study the invariants of the diagonal action of $PGL_n$ on $M_{n\times n}^r$ by conjugation.
To state the theorem, let $\sigma\in S_n$ be written as the product of disjoint cycles $\sigma=(i_1,i_2,\ldots)\cdots (j_1,j_2\ldots)$.
We define $T_{\sigma}:End(W)^{\ot n}\ra K$ by $$T_{\sigma}(M_1\ot M_2\ot\cdots\ot M_n) = tr(M_{i_1}M_{i_2}\cdots)\cdots tr(M_{j_1}M_{j_2}\cdots ).$$
\begin{theorem}\label{tsigma}
The linear maps $\{T_{\sigma}\}_{\sigma\in S_n}$ span the space $$Hom_{PGL(W)}(End(W)^{\ot n},K).$$
\end{theorem}
\end{subsection}
\end{section}

\begin{section}{The variety $X$ and the group $G$}\label{constructingX}
Let $H$ be a finite dimensional semisimple Hopf algebra over $K$.
By Larson-Radford Theorem (see \cite{LR1}) we know that 
$H^*$ is also a semisimple algebra.
The algebra $H$ is thus isomorphic (as an algebra) with $A=\oplus_iEnd(W_i)$ and the algebra $H^*$ is isomorphic (also as an algebra) with $B=\oplus_j End(V_j)$,
where $W_i$ and $V_j$ are the distinct types of irreducible representations of $H$ and of $H^*$ respectively 
(we will assume that $W_1$ is the trivial representation of $H$ and $V_1$ is the trivial representation of $H^*$).
The isomorphisms $H\cong A$ and $H^*\cong B$ induce a linear isomorphism $A^*\cong H^*\cong B$.
In the other direction, a linear isomorphism $A^*\ra B$ will induce a coalgebra structure on $B$, but for most linear isomorphisms we will not get a bialgebra structure on $B$. 

We consider the space of linear transformations $Hom_K(A^*,B)$ as the affine space $A\ot B$. 
We denote by $D$ the determinant polynomial (in order to write the determinant we need to fix a basis to $A$ and to $B$,
but in any case, $D$ is well defined up to a non-zero scalar). 
We can thus identify the Zariski open subset $(A\ot B)_D= \{T\in A\ot B| D(T)\neq 0\}$ with the set of all linear isomorphisms $A^*\ra B$.
We claim the following:
\begin{lemma}
The condition that $T\in (A\ot B)_D$ defines a bialgebra structure on $B$ for which $W_1$ and $V_1$ are the trivial representations is a closed condition.
We denote by $Y\subseteq (A\ot B)_D$ the corresponding closed subset.
The condition that $T\in Y$ defines a Hopf algebra structure is an open condition, given by the non-vanishing of a second polynomial which we denote by $an$
\end{lemma}
\begin{proof}
For $T$ to define a bialgebra structure for which $W_1$ and $V_1$ are the trivial representations, 
we need that $T(\ep_A) = 1_B$, $T^*(\ep_B)=1_A$ (where $T^*:B^*\ra A$ is the dual map), and we need the Hopf axiom to hold in $B$.
The first two conditions are affine equations on $T$, and are therefore clearly closed.
The last condition can be written as an equality between two linear endomorphisms of $B\ot B$:
$$(m_B\ot m_B)T^{\ot 4}(1\ot \tau\ot 1)(\Delta_{A^*}\ot\Delta_{A^*})(T^{-1}\ot T^{-1}) = T^{\ot 2}\Delta_{A^*}T^{-1}m_B$$
where $m_B$ is the multiplication in $B$, $\Delta_{A^*}$ is the comultiplication in $A^*$ (which is the dual of the multiplicative structure of $A$) and $\tau:A^*\ot A^*\ra A^*\ot A^*$ is the natural flip operation.
If we fix a basis for $A$ and for $B$, the entries of $T^{-1}$ can be written as a rational function in the entries of $T$ (with denominators of the form $D^n$),
and the last equation becomes a polynomial equation on the entries of $T$ (once we multiply be a high enough power of $D$).
A bialgebra $H$ is a Hopf algebra if and only if the identity $Id:H\ra H$ is convolution invertible.
For a finite dimensional Hopf algebra, this means that $Id\in H\ot H^*$ should be invertible in the tensor product algebra.
This translates to the fact that a linear isomorphism $T\in A\ot B$ which defines a bialgebra structure 
will define a Hopf algebra structure if and only if it is invertible when considered as an element of $A\ot B$.
But since both $A$ and $B$ are sums of matrix algebras, this can be written as the non-vanishing of a polynomial $an(T)$.
\end{proof}
We thus get a subvariety $X\subseteq A\ot B$ of all linear isomorphisms which define a Hopf algebra structure on $B$.
By the last lemma, $K[X]\cong (K[A\ot B]/I)_{an,D}$ where $I$ is the radical of the ideal generated by the closed conditions in the lemma.
We next ask when do two points in $X$ define isomorphic Hopf algebra structures.
To answer this question, we introduce the group $G=Aut_{alg}(A,\ep_A)\times Aut_{alg}(B,\ep_B)$
(by $Aut_{alg}(A,\ep_A)$ we mean all the algebra automorphisms of $A$ which fix the one dimensional trivial character $\ep_A$, and similarly for $B$).
The group $G$ acts on $A\ot B$ in a natural way. It stabilizes the subvariety $X$, and we claim the following:
\begin{lemma}
Two points $T_1,T_2\in X$ will define isomorphic Hopf algebra structures if and only if they belong to the same $G$-orbit.
\end{lemma}
\begin{proof}
Assume that $T_1,T_2:A^*\ra B$ define isomorphic Hopf algebra structures on $B$.
We will denote the two structures by $B_1$ and $B_2$ respectively.
We thus have a Hopf algebra isomorphism $\beta:B_1\ra B_2$. 
This means that $\beta$ is an automorphism of $B$ as an algebra, and that the linear isomorphism
$$\alpha^*: A^*\stackrel{T_1}{\ra}B_1\stackrel{\beta}{\ra} B_2\stackrel{T_2^{-1}}{\ra}A^*$$
is a coalgebra automorphism (or, alternatively, that the dual map $\alpha:A\ra A$ is an algebra automorphism).
But this is equivalent to the equation $(\alpha^{-1},\beta)(T_1) = T_2$. Since $(\alpha^{-1},\beta)\in G$, we are done.
\end{proof}
The next lemma tells us why we can apply Geometric Invariant Theory to study the orbit space $X/G$:
\begin{lemma}
The stabilizer of each point in $X$ is finite, and therefore all the orbits are closed. 
\end{lemma}
\begin{proof}
If $T\in X$ defines a Hopf algebra $H$, then we can identify between the stabilizer of $T$ and the group $Aut_{Hopf}(H)$ of all Hopf automorphisms of $H$.
Radford proved in \cite{Radford-aut} that this group is finite when $H$ is semisimple and $K$ is of characteristic zero.
Thus, the dimensions of all the orbits is the same as the dimension of $G$, and they are all closed. 
\end{proof}
\begin{remark}
We know, by a theorem of Stefan (see \cite{Stefan}), that the number of orbits of $G$ in $X$ is finite. 
\end{remark}
Finally, we give an explicit description of the group $G$:
\begin{lemma}
 The group $G$ is reductive and fits into a split short exact sequence of the form:
 $$1\ra \ti{G}\ra G\ra \prod_i S_{n_i}\ra 1$$
 where $\ti{G}=\prod_i PGL(W_i)\times\prod_jPGL(V_j)$.
\end{lemma}
\begin{proof}
An algebra automorphism of $A$ will permute the representations of $A$ of the same dimension, and similarly for $B$.
This gives us the surjective homomorphism $G\ra \prod_i S_{n_i}$.
The kernel $\ti{G}$ of this homomorphism will be all the automorphisms which fix the centers of $A$ and of $B$.
By Skolem-Noether Theorem, we know that all such automorphisms are given by conjugation, and therefore we have that $\ti{G}=\prod_i PGL(W_i)\times\prod_jPGL(V_j)$ indeed.
By choosing specific bases for the vector spaces $W_i$ and $V_j$ it is easy to describe a splitting of the surjection $G\ra \prod_i S_{n_i}$.
Finally, since projective general linear groups are reductive, and direct products and finite extensions of reductive groups are again reductive (since the ground field is of characteristic zero), the group $G$ is reductive as well.
\end{proof}
We thus see that two points $T_1$ and $T_2$ in $X$ will be in the same orbit under the action of $\ti{G}$ if and only if there is an isomorphism between the resulting Hopf algebras 
such that the isomorphism between $A$ and $B$ permutes the irreducible representations of $A$ and of $B$ trivially.
In other words, we have the following corollary:
\begin{corollary}\label{orbitstilde}
The orbits of $\ti{G}$ in $X$ correspond to isomorphism types of tuples $(H,W_1,\ldots W_c,V_1\ldots V_d)$ where $H$ is a Hopf algebra, $W_i$ are the irreducible representations of $H$ and $V_j$ are the irreducible representations of $H^*$.
\end{corollary}
\end{section}
\begin{section}{The $\ti{G}$ invariants in $K[X]$}\label{tiG-invariants}
We will study the $G$-invariants in $K[X]$ in two steps.
In this section we will concentrate on the $\ti{G}$-invariants, 
and in the next section we will study the action of the finite group $G/\ti{G}$ on $K[X]^{\ti{G}}$.

Recall first that we have
$K[X] = (K[A\ot B]/I)_{an,D}$, and the action of $G$ on $X$ is induced from a linear action of $G$ on $A\ot B$.
The ideal $I$ is $G$-stable, and the polynomials $an$ and $D^2$ are $G$-invariants.
Since the group $G$ is reductive, the exactness of the sequence of $G$-maps $$0\ra I\ra K[A\ot B]\ra K[A\ot B]/I\ra 0$$
implies the exactness of the sequence $$0\ra I^G\ra K[A\ot B]^G\ra (K[A\ot B]/I)^G\ra 0.$$
In other words, the natural map $K[A\ot B]^G/I^G\ra (K[A\ot B]/I)^G$ is an isomorphism.
We then have $$K[X]^G\cong (K[A\ot B]^G/I^G)_{an,D^2}.$$
We summarize this in the following lemma:
\begin{lemma}
The algebra $K[X]^G$ is generated by the image of the restriction map from the algebra $K[A\ot B]^G$ 
together with $an^{-1}$ and $D^{-2}$.
\end{lemma}
The lemma holds also if we replace $G$ by $\ti{G}$.

In order to find a generating set for $K[X]^{\ti{G}}$ 
it is therefore enough to find a generating set for $K[A\ot B]^{\ti{G}}$.
By lemma \ref{invpgl}, it is enough to study the spaces $Hom_{\ti{G}}((A\ot B)^{\ot n},K)$
(where $n$ is some natural number).
The vector space $A\ot B$ splits as the direct sum of the subspaces $End(W_i)\ot End(V_j)$.
We use the fact that if we have two algebraic groups $G_1$ and $G_2$ acting on 
finite dimensional vector spaces $V_1$ and $V_2$ respectively, then $G_1\times G_2$ acts 
on $V_1\ot  V_2$ in a natural way, and we have a natural isomorphism
\begin{equation}\label{product1}Hom_{G_1\times G_2}(V_1\ot V_2,K)\cong Hom_{G_1}(V_1,K)\ot Hom_{G_2}(V_2,K).\end{equation}
The space $Hom_{\ti{G}}((A\ot B)^{\ot n},K)$ is isomorphic with the direct sum of spaces of the form
$$Hom_{\ti{G}}(End(W_{i_1})\ot End(V_{j_1})\ot\cdots\ot End(W_{i_n})\ot End(V_{j_n}),K).$$
After rearranging the tensor factors we get that this is isomorphic with the following direct sum of all spaces of the form
$$Hom_{\ti{G}}(\bigotimes_{i}End(W_i)^{\ot a_i}\ot 
	\bigotimes_j End(V_j)^{\ot b_j},K)$$
where $\sum_i a_i = \sum_j b_j = n$. 
But this space can be split by using Equation \ref{product1}. 
It is isomorphic with the tensor product
$$\bigotimes_iHom_{PGL(W_i)}(End(W_i)^{\ot a_i},K)\ot \bigotimes_j Hom_{PGL(V_j)}(End(V_j)^{\ot b_j},K).$$
Theorem \ref{tsigma} gives us a description of these spaces.
Indeed, the vector space $Hom_{PGL(W)}(End(W)^{\ot a},K)$ will be spanned by the linear transformations 
$\{T_{\sigma}\}_{\sigma\in S_a}$, where $T_{\sigma}$ is described at the end of Section \ref{prelim}.

We will give now an alternative description of the transformations $T_{\sigma}$.
This will give us a neater description of the generators of $K[X]^{\ti{G}}$.
The transformation $T_{\sigma}$ is constructed using $Tr_W\in End(W)^*$.
If the cycle lengths of $\sigma$ are $c_1,\ldots c_r$, then $T_{\sigma}$ can be described in the following way:
take $Tr_W^{\ot r}$, apply to it $\Delta^{c_1-1}\ot\cdots\ot \Delta^{c_r-1}$
where $\Delta:End(W)^*\ra End(W)^*\ot End(W)^*$ is the dual of the multiplication on $End(W)$,
and apply some permutation on the tensor factors of this result.
This will give us the element $T_{\sigma}$ in $(End(W)^*)^{\ot a}$.

If we trace this back to $K[A\ot B]^{\ti{G}}_n$, we get the following spanning set:
write $\psi_i$ for the the character of $W_i$ and $\phi_j$ for the character of $V_j$ as in Section \ref{prelim}.
Take a tensor product of characters $\psi_{i_1}\ot\psi_{i_2}\ot\cdots\ot \phi_{j_1}\ot\cdots\phi_{j_r}$,
apply to the different tensor factors repeatedly the comultiplications of the coalgebras $End(W_i)^*$ and $End(V_j)^*$ until we get an element in $Hom_{K}((A\ot B)^{\ot n},K)$ and apply a permutation in $S_n\times S_n$ on the result. 
Then the resulting elements are $\ti{G}$-invariant, 
and all the $\ti{G}$-invariant elements are spanned by them.

In order to get the desired invariant polynomial, we just need to evaluate these transformations on $T^{\ot n}\in (A\ot B)^{\ot n}$. This also gives us a concrete description of these invariants in
Hopf algebraic terms. If $T\in A\ot B$ is a point in $X$ which gives us a Hopf algebra structure on $B$,
then we can consider $T$ as the identification between $A^*$ and $B$.
Evaluating $f\ot g\in A^*\ot B^*$ on $T$, will then be the same as $g(f)$
where we identify $f$ with its image in $B$ via $T$.

Therefore, a spanning set for $K[A\ot B]^{\ti{G}}_n$
can be described in the following way:
take a tensor product of characters of $A$ and of $B$, apply the comultiplication 
repeatedly, until we get an element in $(A^*)^{\ot n}\ot (B^*)^{\ot n}$, apply a permutation in $S_n$ to $(A^*)^{\ot n}$
and pair the result with $T^{\ot n}$. We call the resulting invariant
a \textit{basic $\ti{G}$-invariant}. 
Notice that we allow here also reducible characters. 
This will make it easier for us to define $G$-basic invariants in the next section.

By using the isomorphisms $H\cong A$ and $H^*\cong B$, and comparing to Definition \ref{invs}, we see that 
the basic $\ti{G}$-invariants are the same as the $(0,0)$-character basic invariants. 
We thus have the following proposition, which, together with Corollary \ref{orbitstilde} and Theorem \ref{GIT} finishes the proof of Theorem \ref{tildeGinv} 
(see also the proof of Theorem \ref{connection}).
\begin{proposition}
The algebra $K[X]^{\ti{G}}$ is generated by by the basic $\ti{G}$-invariants, up to a localization by $an^{-1}$ and $D^{-2}$.
\end{proposition}
In fact, we have just proved that the algebra $K[X]^{\ti{G}}$ is spanned by elements of the form
$\frac{a}{(an)^iD^{2j}}$ where $a$ is some basic $\ti{G}$-invariant.

Let us see some examples of basic $\ti{G}$-invariants: 
we assume that $H$ is a Hopf algebra with an algebra structure isomorphic with $A$ and coalgebra structure isomorphic with $B^*$.
If $\chi$ is a character of $H$ and $g$ is a character of $H^*$,
then $\chi(g)$ will be a $\ti{G}$-basic invariant.
Another example will be $\chi(g_1g_2) = \chi_1(g_1)\chi_2(g_2)$.
If $\rho$ is another character of $H$ and $h$ is another character of $H^*$,
we also have the $\ti{G}$-basic invariant
$$\chi_1(g_1)\chi_2(h_2)\chi_3(h_3)\rho_1(g_2)\rho_2(g_3)\rho_1(h_1).$$
\end{section}

\begin{section}{The $G$-invariants in $K[X]^G$ and a proof of Theorems \ref{main1} and \ref{main2}}\label{G-invariants}
In this section we use our study of the algebra $K[X]^{\ti{G}}$ from the previous section
in order to describe a generating set for the algebra $K[X]^G$.
We define a \textit{$G$-basic invariant} to be a $\ti{G}$-basic invariant,
in which all the characters which appear are the characters $\la\in A^*$ of the regular representation of $A$,
and the character $\l\in B^*$ of the regular representation of $B$.
These characters can be written as $\la = \sum_i dim(W_i)\psi_i$ and $\l=\sum_j dim(V_j)\phi_j$.
In other words, for a given Hopf algebra $H$ such that $H\cong A$ and $H^*\cong B$ as algebras, these are going to be the same as the $(0,0)$-basic invariants from Definition \ref{invs}.
Since the group $G/\ti{G}$ acts by permuting characters of the same dimension, 
and since all the characters of the same dimension appear with the same multiplicity in $\l$ and in $\la$, 
it is easy to see that the $G$-basic invariants will be invariant under the action of the quotient $G/\ti{G}$,
and are therefore $G$-invariant. We claim the following proposition:
\begin{proposition}\label{G-invproposition}
The $G$-basic invariants span $K[X]^G$ up to localization by $an$ and $D^{2}$.
\end{proposition}
\begin{remark}
It is worth mentioning that this proposition will not be true for $K[A\ot B]^G$.
We will use here explicitly some Hopf algebra identities from Section \ref{prelim}.
\end{remark}
\begin{proof}
Due to Lemma \ref{invfinite} we know that the map $$K[X]^{\ti{G}}\ra K[X]^G$$ 
$$f\mapsto \sum_{g\in G/\ti{G}}g(f)$$ is onto.
We have seen in the last section that $K[X]^{\ti{G}}$ 
is spanned (up to negative powers of $an$ and $D$) by $\ti{G}$-basic invariants
in which all the characters are irreducible. 
Let then $P$ be such a $\ti{G}$-basic invariant. 
We can write $P=T^{\ot n}(T_1T_2\cdots T_s(\psi_{i_1}\ot \ldots \psi_{i_r}\ot\phi_{j_1}\ot\cdots\ot \phi_{j_l}))$
where $T_i$ are operations of comultiplication on $A^*$, comultiplication on $B^*$ and the action of the symmetric group, 
$\psi_i$ are irreducible characters of $A$ and $\phi_j$ are irreducible characters of $B$.
We need to show that $$\sum_{g\in G/\ti{G}}g(P) = \sum_{g\in G/\ti{G}}T^{\ot n}(T_1T_2\cdots T_s(\psi_{g(i_1)}\ot\cdots \ot\psi_{g(i_r)}\ot\phi_{g(j_1)}\ot\cdots\ot\phi_{g(j_l)}))= $$ $$
T^{\ot n}(T_1T_2\cdots T_s\sum_{g\in G/\ti{G}} \psi_{g(i_1)}\ot\cdots \ot\psi_{g(i_r)}\ot\phi_{g(j_1)}\ot\cdots\ot\phi_{g(j_l)})$$
is a sum of $G$-basic invariants.
Because $G/\ti{G}$ 
is the product of all symmetric groups on the irreducible representations of $A$ and of $B$ of the same dimension (besides the trivial one dimensional representations, but since we can express these representations as 
$\la_1(\l)\la_2$ and $\la(\l_1)\l_2$ this makes no real difference),
Equation \ref{idem-char} shows that we can express the tensor product of the characters by the tensors $x^{a_1,\ldots, a_n}_{d,n}$ and $\l$ and $\la$. 
But since the tensors $x^{a_1,\ldots, a_n}_{d,n}$ themselves can be obtained from $\l^{\ot m}$ and $\la^{\ot m'}$ 
for some $m$ and $m'$ 
by applying permutations, multiplications and comultiplications, we get that 
also the expression $\sum_{g\in G/\ti{G}} g(P)$ can be obtained from $l^{\ot m}\ot \la^{\ot m'}$ 
by applying comultiplication and the action of the symmetric group.
This implies that $\sum_{g\in G/\ti{G}} g(P)$ is a sum of $G$-basic invariants, as desired.
\end{proof}
This gives us a set of generators for $K[X]^G$ which can be described nicely in combinatorial terms.
It still does not give us a full description of the algebra $K[X]^G$ since we do not know what are 
all the relations between these generators. 
We can divide the relations the $G$-basic invariants satisfy into two groups:\\
1. The relations arising from relations among the same invariants in the algebra $K[A\ot B]^G$.\\
2. The relations arising from the ideal $I^G$.\\
Procesi has studied the relation between the generators of the algebra $K[End(W)^r]^{PGL(W)}$.
He showed that all the relations can be deduced from the Cayley-Hamilton Theorem,
and he also gave a bound on the number of generators which will suffice to generate the entire algebra.

Trying to study $K[X]^G$ by studying all the relations of the two types may turn difficult.
We shall use the invariants to study Hopf algebras, only without studying specifically the structure of $K[X]^G$.
Notice that the description of the $G$-basic invariants 
is somewhat uniform: it does not depend on the dimension of $H$
or the dimensions of the irreducible representations of $H$.
Indeed, the expression $\la(\l_1\l_2)$, for example, makes sense in any finite dimensional Hopf algebra.
We shall therefore call the $G$-basic invariants simply basic invariants from now on.
Moreover, we have the following proposition, which finishes the proof of Theorem \ref{main1}:
\begin{proposition}\label{connection}
Two Hopf algebras are isomorphic if and only if all their basic invariants are equal.
\end{proposition}
\begin{proof}
On the one hand, the basic invariants are invariants of the isomorphism type of the Hopf algebra,
and therefore if $H_1\cong H_2$ then they have the same basic invariants.	
On the other hand, if $H_1$ and $H_2$ have the same basic invariants then in particular their dimensions are equal, since $\la(\l)=dim(H)$.
Moreover, by considering the invariants $\la({c^{1,d}_H})$ for different $d$'s we see 
that the number of irreducible representations of dimension $d$ in $H_1$ and in $H_2$ is the same
(and the same holds for $H_1^*$ and $H_2^*$).
We can thus consider $H_1$ and $H_2$ as points in the variety $X$ (for a suitable choice of dimensions of irreducible representations).
Then, since all the $G$-invariant functions on $X$ receive the same value on $H_1$ and $H_2$
it must hold that $H_1$ and $H_2$ lie in the same $G$-orbit (by Theorem \ref{GIT}), 
and they are therefore isomorphic.
\end{proof}
The next proposition is a more detailed reformulation of Theorem \ref{main2}
\begin{proposition}
Let $H$ be a semisimple Hopf algebra. 
Consider the field extension $\Q\subseteq K_0$ generated over $\Q$ by all the basic invariants of $H$.
Then $K_0$ is a finite extension of $\Q$ (i.e. $K_0$ is a number field),
and if we denote by $\Ga$ the absolute Galois group of $\Q$, then
$$stab_{\Ga}([H]) = \{\ga\in\Ga| \,^{\ga} H\cong H\}= stab_{\Ga}(K_0) =\{\ga\in\Ga|\forall x\in K_0 \ga(x) = x \}$$
where $^{\ga}H$ is received from $H$ by twisting all its structures constants by $\ga$. 
\end{proposition}
\begin{proof}
The variety $X$, the group $G$ and the action of $G$ on $X$ are defined already over $\Q$.
By abuse of notations, we will identify $X$ and $G$ with the variety and algebraic group defined over $\Q$ and over $\wb{\Q}$.
Since there are only finitely many orbits in $X$, we have that 
$\wb{\Q}[X]^{G} \cong \wb{\Q}^m$ where $m$ is the number of orbits over $\wb{\Q}$,
and similarly $K[X]^G\cong  K^{m'}$ where $m'$ is the number of orbits over $K$.
But it then holds that $\wb{\Q}[X]^{G}\ot_{\wb{\Q}}K \cong K[X]^G$, and therefore $m=m'$. 
It follows that the equations defining the orbit of $H$ in $X$ 
are already defined over $\wb{\Q}$. Since $\wb{\Q}$ is algebraically closed, it follows that the
orbit of $H$ has a point over $\wb{\Q}$, and therefore $H$ is defined over $\wb{\Q}$,
and all its basic invariants are contained in $\wb{\Q}$.
Since $H$ has only finitely many structure constants, it is easy to see that $H$ will be defined
over some finite extension of $\Q$, and therefore all the basic invariants of $H$
will be contained in some finite extension of $\Q$.

For the second claim, let $\ga\in\Ga$. Then if $a\in K_0$ is a basic invariant of $H$,
the corresponding basic invariant of $^{\ga}H$ will be $\ga(a)$.
Since two Hopf algebras are isomorphic if and only if they have the same basic invariants,
we see that $^{\ga}H\cong H$ if and only if $\ga$ fixes $K_0$ pointwise, as desired.
\end{proof} 
\begin{remark}
I first learned that a finite dimensional semisimple Hopf algebra over $K$ is already defined over some finite extension of $\Q$ from Juan Cuadra.
This fact seems to be well known. I include here a proof due to the lack of reference. 
\end{remark}

\end{section}

\begin{section}{Invariant subspaces. A proof of Theorem \ref{main3}}\label{invsspaces}
We fix now a Hopf algebra $H$ with an algebra structure $A$ and a coalgebra structure $B^*$.
As usual, we think of $H$ as a point $T\in X\subseteq A\ot B$,
and we think of $Inv^{i,j}$ as a subspace of $(B^*)^{\ot i}\ot (A^*)^{\ot j}$. 
By this identification, $T\in A\ot B\cong H\ot H^* = (H^*\ot H)^*$ can be identified with the evaluation map $H^*\ot H\ra K$.
We would like to prove that the pairing $Inv^{i,j}\ot Inv^{j,i}\ra K$ is non-degenerate.
In the course of the proof we will need to use the following lemma:

\begin{lemma}\label{replacement}
Let $W\in Hom_G(A^{\ot (j+m)}\ot B^{\ot (i+m)},K)$.
Then $(T^{\ot m}\ot Id)(W)$ is contained in $Inv^{i,j}\subseteq (B^*)^{\ot i}\ot (A^*)^{\ot j}$
\end{lemma}
\begin{proof}
The proof of the lemma follows the line of the proof of Proposition \ref{G-invproposition}.
Since $(T\ot Id)(Inv^{i,j})\subseteq Inv^{i-1,j-1}$ by definition of $Inv^{i,j}$,
it is enough to prove that $Hom_G(A^{\ot j+m}\ot B^{\ot i+m},K)\subseteq Inv^{i+m,j+m}$.
The space $Hom_G(A^{\ot j+m}\ot B^{\ot i+m},K)$ will be spanned by elements of the form
$$\sum_{g\in G/\ti{G}}T_1\cdots T_s(\psi_{g(i_1)}\ot\cdots \ot \psi_{g(i_r)}\ot \phi_{g(j_1)}\ot\cdots\ot \phi_{g(j_l)})$$
where the $T_k$ operators are either given by comultiplication on $A^*$, comultiplication on $B^*$ or the action of the symmetric group. 
Since $Inv^{i,j}$ is closed under the action of $S_i\times S_j$, and since
$(\Delta\ot Id)(Inv^{i,j})\subseteq Inv^{i+1,j}$ 
we see that we only need to show that sums of the form 
$$\sum_{\substack{(i_1,\ldots,i_n):\\ \psi_{i_j}(1)=d\, , |\{\psi_{i_1},\ldots, \psi_{i_n}\}|=n}} \psi_{i_1}^{\ot a_1}\ot \psi_{i_2}^{\ot a_2}\ot\cdots\ot \psi_{i_n}^{\ot a_n}$$ 
where $\psi_i$ are irreducible characters are contained in $Inv^{0,N}$ for $N=\sum_i a_i$.
This now follows easily from Corollary \ref{ekses} and Equation \ref{idem-char}.
\end{proof}
\begin{proof}[Proof of Theorem \ref{main3}]
We take an element $x\in Inv^{j,i}$ which is perpendicular to $Inv^{i,j}$, and show that it must be zero. 
To say that $x=0$ is equivalent to saying that a certain set of polynomials on $X$ vanish when applied to $T$.
Let us denote the ideal generated by these polynomials by $J\triangleleft K[X]$.
The element $x$ will thus be zero if and only if for every element $y\in (B^*)^{\ot i}\ot (A^*)^{\ot j}$ we have that
$\langle y,x\rangle=0$. We denote this equation by $f_{y}(x)$. Thus $J=(f_y)$. We claim the following:
\begin{lemma} 
For $g\in G$ we have that $g\cdot f_y = f_{g\cdot y}$.  
\end{lemma}
\begin{proof}
We have a (permuted) tensor product $t\in ((A\ot B)^*)^{\ot n}\ot (B^*)^{\ot j}\ot (A^*)^{\ot i}$ of sums of iterated comultiplications of copies of the regular characters of $A$ and $B$ such that
$$x=(T^{\ot n}\ot Id)(t).$$ This is true for any basic invariant, and we can show that it holds also for general invariants, by taking a large enough $n$.
Then $\langle y,x\rangle = T^{\ot n+i+j}(t\ot y)$ (after rearranging the tensor factors). 
We have that $$(g\cdot f_y)(x) = f_y(g^{-1}x) = (g^{-1}\cdot T)^{\ot n+i+j}(t\ot y) = T^{\ot n+i+j}(t\ot g\cdot y)$$  
where in the last step of the computation we have used the fact that $t$ is $G$-invariant.
This implies that $g\cdot f_y= f_{g\cdot y}$ as desired.
\end{proof}
	
From the proof of the lemma we also see that $G\cdot J = J$, since $G$ stabilizes a generating set of $J$, namely $\{f_y\}$ for $y\in (B^*)^{\ot i}\ot (A^*)^{\ot j}$.
We shall denote this vector space by $V$. Thus $V$ is a $G$-representation, isomorphic with $(B^*)^{\ot i}\ot (A^*)^{\ot j}$.
Since $J$ is stable under the action of $G$, we know that $V(J)\subset X$ is also stable under the action of $G$.
We therefore have that $x=0$ if and only if $V(J)$ contains the orbit $\Ow_T$ of $T$ in $X$, 
and this happens if and only if the image of $J$ is zero in $K[\Ow_T]$.
It follows from Theorem \ref{GIT} that the image of $J^G$ in $K[\Ow_T]^G$ is either zero or the entire ring (because $J$ defines a $G$-stable closed subset of $K[\Ow_T]$, and such a subset can only be 
the empty set or $\Ow_T$ itself).
Therefore, $x=0$ if and only if the image of $J^G$ is zero in $K[\Ow_T]$.
We thus need to find the invariants in $J=K[X]\cdot\{f_y\}$.
As before, this boils down to calculating the $G$-invariants in $$(A^*)^{\ot m}\ot (B^*)^{\ot m}\ot V\cong
(A^*)^{\ot m}\ot (B^*)^{\ot m}\ot (B^*)^{\ot i}\ot (A^*)^{\ot j}\cong $$ $$\cong (A^*)^{\ot m+j}\ot (B^*)^{\ot m+i}$$ for every $m$. 
If $W\in Hom_G(A^{\ot m+j}\ot B^{\ot m+i},K)$ then the value of the polynomial $p_W$ resulting from applying $W$ on $T$ can be described as follows:
first apply $W$ to $T^{\ot m}$ to get an element $\ti{W}\in Hom(A^{\ot j}\ot B^{\ot i},K)$
which belongs to $Inv^{i,j}$, by Lemma \ref{replacement}.
We have $p_W(T) = \langle \ti{W},x\rangle$ where the pairing is done using $T$, as usual.
Since $x$ is perpendicular to $Inv^{i,j}$ we get that $p_W(T)=0$.
But this means that all polynomials in $J^G$ vanish on $T$, and this implies that $x=0$, as desired.
\end{proof}
Theorem \ref{main3} can be proved verbatim also for the action of $\ti{G}$.
This gives us the first part of Theorem \ref{main4.5}:
\begin{proposition}\label{dualityprop2}
The pairing $\ti{Inv}^{i,j}\ot\ti{Inv}^{j,i}\rightarrow K$ is non-degenerate. 
\end{proposition}
\end{section}

\begin{section}{Construction of symmetric monoidal categories out of invariant subspaces.
A proof of Theorem \ref{main4}}\label{Tannaka}
In this section we will construct a symmetric monoidal category out of the spaces $Inv^{i,j}$.
We will then use Tannaka Reconstruction Theorem for symmetric monoidal categories to study this category,
and we shall prove that $Inv^{i,j}=(H^{i,j})^{Aut_{Hopf}(H)}$.
Let then $\C_1$ be the following category: the objects of $\C_1$ are the vector spaces $H^{i,j}$.
The morphism spaces are given by $$Hom_{\C_1}(H^{i,j},H^{a,b}) = Inv^{j+a,i+b}.$$
Since $Inv^{j+a,i+b} \subseteq H^{j+a,i+b} \cong Hom_K(H^{i,j},H^{a,b})$,
we can define composition of morphisms just as composition of linear maps.
The way we have defined the spaces $Inv^{i,j}$ ensures us that this composition will be well defined.
Since the identity map in $H$ can be written by using $\l$ and $\la$, 
we know that $\C_1$ has identity maps as required. 
Moreover, all the $Hom$-sets in $\C_1$ are finite dimensional $K$-vector spaces,
and composition of morphisms is $K$-bilinear.

The category $\C_1$ is also a rigid monoidal category (rigidity means that each objects has a dual): the tensor product is the same as that in vector spaces,
that is: $H^{i,j}\ot H^{a,b} = H^{i+a,j+b}$, the tensor unit is given by $H^{0,0}$, 
and the dual is given by $(H^{i,j})^* = H^{j,i}$.
Since we have used the action of the symmetric groups on tensor products in the construction of the spaces $Inv^{i,j}$
we get that $\C_1$ is a rigid symmetric monoidal category.
We also have an obvious symmetric monoidal functor $F_1:\C_1\ra Vec_K$ which sends $H^{i,j}$ to $H^{i,j}$

We would like to construct an abelian rigid symmetric monoidal category out of $\C_1$ in order to apply Tannaka Reconstruction Theorem. 
For this, we define $\C_2$ to be the additive envelope of $\C_1$: objects of $\C_2$ are 
formal direct sums of objects of $\C_1$, and morphisms are given by suitable matrices of morphisms in $\C_1$.
The functor $F_1$ can be extended in a natural way to a functor $F_2:\C_2\ra Vec_K$
(this follows easily from the fact that $Vec_K$ has direct sums). 
We define the category $\C$ to be the Karoubian envelope of $\C_2$: objects of $\C$ will be pairs 
$(P,p)$ where $P$ is an object of $\C_2$ and $p:P\ra P$ satisfies $p^2=p$,
and morphisms $f:(P,p)\ra (Q,q)$ are the morphisms $f:P\ra Q$ which satisfy $f=qfp$.
Intuitively, we think of $(P,p)$ as the image of $p:P\ra P$.
Again, since all projections in $Vec_K$ have images, the functor $F_2$ can be extended naturally to a functor $F:\C\ra Vec_K$ which sends $(P,p)$ to the vector space $Im(p)$.

We would like to prove that the category $\C$ is abelian. 
We begin with proving the following lemma:
\begin{lemma}\label{projections1}
Let $f:C\ra D$ be a morphism in $\C_2$.
Then $f$ has a kernel and a cokernel in $\C$, where we consider $f$ as a morphism 
$f:(C,1_C)\ra(D,1_D)$.
\end{lemma}
\begin{proof}
It will be enough to prove that $f$ has a kernel. 
Proving that $f$ has a cokernel can be done in a dual way.
So let $f:C\rightarrow D$ be a morphism in $\C_2$.
Consider the object $E=C\oplus D$. 
We have an endomorphism $\ti{f}:E\rightarrow E$ given symbolically by $(c,d)\mapsto (0,f(c))$.
If $\ti{f}$ has a kernel in $\C$, then it will be of the form $Ker(f)\oplus D$.
Therefore, since projections have kernels and images in $\C$, 
if we will prove that $\ti{f}$ has a kernel, then we will know that $f$ has a kernel.
So we can reduce to the case $C=D$.

Let us consider now the finite dimensional $K$-algebra $$R:=End_{\C_2}(C,C)\cong Hom_{\C_2}(K,C\ot C^*).$$
We have a canonical map $tr:R\rightarrow K$ induced by the evaluation $C\ot C^*\rightarrow K$.
The algebra $R$ can be thought of as a subalgebra of $End_K(F(C))$. 
As such, the functional $tr$ is the usual trace of endomorphisms of $F(C)$ restricted to $R$.
We know that if we take $C=H^{i,j}$ then the pairing 
\begin{equation}\label{pairing} R\ot R\stackrel{m_R}{\ra}R\stackrel{tr}{\ra} K\end{equation} will be non-degenerate.
This follows directly from the fact that the pairing $Inv^{i+j,i+j}\ot Inv^{i+j,i+j}\ra K$
is non-degenerate (by Theorem \ref{main3}).
Now, since every object of $\C_2$ is a direct summand of a direct sum of objects of the form $H^{i,j}$
we see that Equation \ref{pairing} will give us a non-degenerate pairing for every $C$.
But this is equivalent to $R$ being semisimple. 
Since $R$ is semisimple, and $K$ is algebraically closed, we know by Wedderburn's Theorem that $R$ is isomorphic with a product of matrix algebras.
This means that $f$ can be written as the composition $f=rp$ where $r\in R$ is invertible, and $p\in R$ is a projection. 
We can identify between $Ker(f)$ and $Ker(p)$.
Since $p$ is a projection, it has a kernel in $\C$ and we are done.
\end{proof}
The next lemma we need will relate invertibility of morphisms in $\C$ and in $Vec_K$. 
We claim the following:
\begin{lemma}
Let $f:C\ra D$ be a morphism in $\C$. Then $f$ is invertible if and only if $F(f):F(C)\ra F(D)$
is invertible in $Vec_K$.
\end{lemma}
\begin{proof}
We begin with the case where $F(C)$ and $F(D)$ are one dimensional. 
In this case  $Hom_{\C}(C,D)\cong Hom_{\C}(K,C^*\ot D)$ is one dimensional (and spanned by $f$), and $Hom_{\C}(D,C)\cong Hom_{\C}(K,D^*\ot C)$ is at most one dimensional.
But we know that $(C^*\ot D)^*\cong D^*\ot C$, and therefore the pairing 
$Hom_{\C}(K,C^*\ot D)\ot Hom_{\C}(K,D^*\ot C)\rightarrow K$ is non-degenerate (by the same argument used in Lemma \ref{projections1}).
This implies that there exists a morphism $g:D\rightarrow C$ such that $gf\neq 0$ and $fg\neq 0$.
By changing $g$ if necessary, we can assume that $gf = 1_C$ and $fg=1_D$, so $f$ is invertible.

Assume now that $dim_K F(C)=dim_K F(D)=n$ (the dimensions are the same, since $F(f)$ is an isomorphism).
Then consider $\bigwedge^n f: \bigwedge^n C\rightarrow \bigwedge^n D$.
(since $\C$ is a symmetric monoidal category in which projections have kernels
we can freely talk about $\bigwedge^n C$: it will be the image of the idempotent $\frac{1}{n!}\sum_{\sigma\in S_n}(-1)^{\sigma}\sigma$ in $End_{\C}(C^{\ot n})$.
Since $F$ is a symmetric monoidal functor we are guaranteed that $F(\bigwedge^n C) \cong \bigwedge^n F(C)$. A similar statement holds for $D$).
This is an isomorphism between one dimensional spaces, and is therefore invertible.
Now the inverse of $f$ can be written as the following composition:
$$D\rightarrow D^{\ot n}\ot (D^*)^{\ot (n-1)}\rightarrow \bigwedge^n D\ot (D^*)^{\ot (n-1)}\rightarrow \bigwedge^n C\ot (D^*)^{\ot (n-1)}\rightarrow C$$
where the first map is the coevaluation on $D^{\ot (n-1)}$, the second map is the projection $D^{\ot n}\rightarrow \bigwedge^n D$,
the third map is the composition of $(\bigwedge^n f)^{-1}$ with $(f^*)^{\ot (n-1)}$ and the last map is the composition of the inclusion $\bigwedge^n C\rightarrow C^{\ot n}$
with evaluation on $C^{\ot (n-1)}$ (the last claim is just a categorical formulation of Cramer Rule).
This shows that $f$ has an inverse in $\C$, and we are done.
\end{proof}
We can now prove the following proposition:
\begin{proposition}
The category $\C$ is abelian
\end{proposition}
\begin{proof}
We begin by proving that any morphism in $\C$ has a kernel and cokernel.
By a duality argument, it will be enough to prove that if $f:(C,p)\ra (D,q)$ is a morphism in $\C$ then it has a kernel.
We can consider $f$ as a morphism in $\C_1$ which satisfies $f= qfp$.
Then we have seen that $\ti{f}:(C,1_C)\ra (D,1_D)$ has a kernel $Ker(\ti{f})$ in $\C$.
A direct verification shows that $p$ induces an endomorphism $\ti{p}:Ker(\ti{f})\ra Ker(\ti{f})$ which is also a projection, 
and the kernel of $1-\ti{p}$ will be the desired kernel of $f$ (we use here the fact that $1-\ti{p}$ is a projection, and projections have kernels in $\C$).

So we see that all morphisms in $\C$ have kernels and cokernels in $\C$.
In order to prove that $\C$ is indeed abelian, we need to prove that if $f:C\ra D$ is a monomorphism (epimorphism)
then the induced map $C\ra Ker(Coker(f))$ ($Coker(Ker(f))\ra D$) is an isomorphism.
We will concentrate on the case where $f$ is a monomorphism. 
By construction of the functor $F$ we know that if $p:C\ra C$ is a projection,
then $F(Ker(p)) = Ker(F(p))$.
By the proof of Lemma \ref{projections1} we see that $F(Ker(g))\cong Ker(F(g))$ in a natural way for every morphism $g$ in $\C$.
But then we have that after applying $F$ to $C\ra Ker(Coker(f))$ 
we get the map $F(C)\ra Ker(Coker(F(f)))$ which is an isomorphism.
We have seen that this implies that the original map in $\C$ is an isomorphism, so we are done.
\end{proof}

The category $\C$ has therefore a very rich structure: it is a rigid symmetric monoidal $K$-linear category.
Moreover, we have a symmetric monoidal functor $F:\C\ra Vec_K$.
By construction the functor $F$ is faithful (that is- the map $Hom_{\C}(C,D)\ra Hom_{Vec_K}(F(C),F(D))$ is injective)
and exact (this follows from the fact that $F$ preserves kernels and cokernels).
Tannaka Reconstruction Theorem now tells us the following:
\begin{theorem}(see Theorem 2.11 in \cite{DM})
Let $\C$ and $F$ be as above. Let $A=Aut_{\ot}(F)$.
Then for every $C\in\C$ the vector space $F(C)$ is an $A$-representation in a natural way,
and the functor $\ti{F}:\C\ra Rep_K-A$ is an equivalence of symmetric monoidal $K$-linear categories.
\end{theorem}
The natural way in which $F(C)$ is an $A$-representation is the following:
every $a\in A$ is an isomorphism $a:F\ra F$. In particular, we will get an invertible map $a_C:F(C)\ra F(C)$
and this will give us an $A$-representation structure on $F(C)$.
In order to apply the theorem, we need to describe the group $A=Aut_{\ot}F$.
If $a\in A$, then the action of $a$ on all $F(C)$ can be deduced by its action on $F(H^{1,0})$.
This follows easily from the fact that $F$ is a monoidal functor.
We can thus consider $A$ as a subgroup of $GL(H)$.
Now, since $Hom_{\C}(K,H^{i,j}) = Inv^{i,j}$,
we have that $Inv^{i,j} = (H^{i,j})^A$.
But the multiplication and comultiplication of $H$ can be written as elements of 
$Inv^{1,2}\subseteq H^{1,2}$ and $Inv^{2,1}\subseteq H^{2,1}$ respectively. 
This implies that every $a\in A$ preserves the algebra and coalgebra structure of $H$,
and therefore $A\subseteq Aut_{Hopf}(H)$.
On the other hand, if we have a Hopf automorphism $a$ of $H$,
then it is easy to see that it fixes $Inv^{i,j}$ pointwise for every $i$ and $j$,
and a careful examination shows that it induces an automorphism of $F_1$, $F_2$ and $F$.
We thus have that $A=Aut_{Hopf}(H)$.
An immediate corollary of this discussion is Theorem \ref{main4}:
we have $Inv^{i,j} = (H^{i,j})^{Aut_{Hopf}(H)}$.
Notice that instead of constructing the category $\C$ using the subspaces $Inv^{i,j}$,
we could have used the subspaces $\ti{Inv^{i,j}}$.
This would result in a category with more morphisms, which is equivalent to the category 
of $Aut_{Hopf}^0(H)$-representations, where $Aut_{Hopf}^0(H)\subseteq Aut_{Hopf}(H)$ is the subgroup of all Hopf automorphisms of $H$ 
which fix all the irreducible characters and cocharacters of $H$. In particular we get Theorem \ref{main4.5}: 
$\ti{Inv^{i,j}} = (H^{i,j})^{Aut^0_{Hopf}(H)}$.

The category $\C$ constructed here can be constructed more generally for any algebraic structure, not just for Hopf algebras.
This construction is carried out in \cite{Meir1} (the construction there is more general, and gives a category defined over $K_0$ instead of over $K$). 
The assumption that the pairing between $Inv^{i,j}$ and $Inv^{j,i}$ is non-degenerate is not necessary for the construction of the category,
but it is necessary in order to to prove here the equality $Inv^{i,j} = (H^{i,j})^{Aut_{Hopf}(H)}$.
Moreover, the construction in \cite{Meir1} can also be used to construct a ``generic form'' of $H$ over a finitely generated commutative $K_0$-algebra.
\end{section}
\begin{section}{Finiteness of the number of orders}\label{finiteorders}
Let now $L\subseteq K$ be a number field.
Assume that $H$ is a semisimple Hopf algebra defined over $L$.
A \textit{Hopf order} of $H$ is a finitely generated $\Ow_L$-submodule $R$ of $H$ which is a Hopf algebra over $\Ow_L$, such that 
the canonical map $R\ot_{\Ow_L}L\ra H$ is an isomorphism of Hopf algebras.
In \cite{CM1} and \cite{CM2} Juan Cuadra and the author studied orders of Hopf algebras by means of the character theory 
of $H$ and $H^*$. A special role is played by the $\ti{G}$-basic invariants in $\ti{Inv^{1,0}}$ and $\ti{Inv^{0,1}}$.
The idea is the following:
If $R$ is a Hopf order of $H$ then $R^{\star} = \{f\in H^*| f(R)\subseteq \Ow_L\}$ 
is a Hopf order of $H^*$. It holds that $(R^{\star})^{\star}=R$.
All the characters of $H$ are contained in $R^{\star}$ 
and all characters of $H^*$ are contained in $(R^{\star})^{\star}=R$.
This implies that all basic invariants are contained in $R$
(and if all representations of $H$ and of $H^*$ are realizable over $L$,
then also all the $\ti{G}$-basic invariants are contained in $R$).
From the last section we know that the basic invariants span the subspace of $Aut_{Hopf}(H)$-invariants.
In this section we shall use this, together with the theorem of Larson, about the finiteness of $Aut_{Hopf}(H)$,
to prove that $H$ has at most finitely many Hopf orders.
\begin{proof}[Proof of Theorem \ref{main5}]
We write $Aut_{Hopf}(H) = A$ as before.
Consider the commutative $L$-algebra $C=L[H]$. 
If $H$ has a basis $\{h_1\ldots,h_d\}$ and $H^*$ has a dual basis $\{h^1,\ldots,h^d\}$
then this algebra can be written as a polynomial algebra in the indeterminates $h^i$.
The group $A$ acts on $C$, and $C$ is integral over $C^A$
(to see why this is true, consider for $c\in C$ the polynomial $\prod_{a\in A}(x-a(c))$).
Let $D\subseteq C^A$ be the sub $\Ow_L$-algebra generated by the images of the basic invariant in $Inv^{0,n}$
for different values of $n$ in $C$. In particular, since we know that $Inv^{0,n} = ((H^*)^{\ot n})^A$,
it follows that for every $c\in C^A$ there exists an $m\in\Z$ such that $mc\in D$.
By a similar argument, this implies that for every $i$ there exists an $m_i\in\Z$ such that 
$m_ih^i$ is integral over $D$ (just use the integrality equation for $h^i$ over $C^A$).
By replacing $h^i$ with $m_ih^i$ we can assume, without loss of generality, that the elements $h^i$ themselves are integral over $D$.

Let now $R$ be a Hopf order of $H$.
Assume that $x=\sum_i t_ih_i\in R$. We would like to prove that $t_i\in \Ow_L$ for every $i$.
For this, we write the integrality equation for $h^i$:
$$(h^i)^n + (b_1)(h^i)^{n-1}+\cdots+ b_n = 0$$
where $b_i\in D$.
By evaluating this equation on $x$ we get:
$$t_i^n + b_1(x)t_i^{n-1} + \cdots + b_n(x) = 0.$$
But since $b_j\in D$ and since we know that when we evaluate basic invariants on elements of $R$
we get elements of $\Ow_L$, we get that $b_j(x)\in \Ow_L$.
Thus, $t_i$ is integral over $\Ow_L$ and is therefore contained in $\Ow_L$.

We therefore conclude that $R$ is contained in $M=\oplus_i\Ow_L h_i$.
In a similar way we can find an $\Ow_L$-submodule $N$ of $H^*$ of maximal rank
such that $R^{\star}\subseteq N$.
It then follows that $N^{\star}\subseteq R\subseteq M$. 
Since both $M$ and $N^{\star}$ are finitely generated $\Ow_L$ modules of the same rank,
and since $\Ow_L$ is a number field, the quotient $M/N^{\star}$ is finite.
It thus have only finitely many subgroups, and therefore $H$ has at most finitely many Hopf orders.
\end{proof}
\begin{remark}
The proof gives us a concrete upper and lower bound for Hopf orders of $H$.
If we pass to a finite extension of $L$ we can assume that all 
representations of $H$ and of $H^*$ are realizable over $L$,
and get a tighter bound, using the $\ti{G}$-basic invariants
instead of the basic invariants.
In many cases it holds that $Aut^0_{Hopf}(H)=1$ (e.g. for group algebras),
and then we automatically get an upper and lower bound for orders, without the construction of the commutative algebra in the proof here.
Nevertheless, there are many examples for Hopf algebras $H$ with $Aut^0_{Hopf}(H)\neq 1$.
One can construct such an example in the following way: Let $\ti{H}$ be a finite dimensional semisimple Hopf algebra, and assume that $g\in \ti{H}$ is a non-central group like element.
Denote by $n$ the order of $g$.
Consider the Hopf algebra $H:=K\langle \sigma|\sigma^n=1\rangle \ltimes \ti{H}^*$, where the action of $\sigma$ is given by the dual action of conjugation by $g$ ($\sigma$ is a group like element).
Then $H^*$ has a group like element $\ti{g}$ given by $\ti{g}(\sigma^i\otimes f) = f(g^{-1})$.
Moreover, it is easy to show that conjugation by $\sigma$ is the same as the dual of conjugation by $\ti{g}$
(and since it is given by conjugation on $H$ and on $H^*$, it fixes all the irreducible characters).
Conjugation by $\sigma$ thus defines a non-trivial element in $Aut_{Hopf}^0(H)$. 
\end{remark}
\end{section}

\begin{section}{Examples: Invariants of group algebras}\label{ex1}
In this section we shall study the basic invariants for the specific example of group algebras.
This example was described in the paper \cite{dks}.
We give some more details and an alternative description of the invariants here.

Let then $H=KG$ be a group algebra of a finite group $G$ of order $n$.
In this example, we have $\l=\sum_{g\in G}g$ and $\la=ne_1$
where $e_1$ is the idempotent which receives 1 on the identity element of $G$ and zero on all the rest
(we identify here the dual Hopf algebra $(KG)^*$ with the algebra of functions on $G$).
We have $\Delta^{t-1}(\l) = \sum_{g\in G}g^{\ot t}$
and $\Delta^{t-1}(e_1) = \sum_{g_1g_2\cdots g_t=1}e_{g_1}\ot e_{g_2}\ot\cdots\ot e_{g_t}$.
If we take $\la^{\ot a}\ot \l^{\ot b}$, apply comultiplication repeatedly, permute the tensor factors and pair the two sides,
we will get the number of solutions to $a$ equations in $b$ variable times $n^a$.
For example, $\la_1(\l_1)\la_2(\l_2)\cdots \la_t(\l_t)$ will be $n$ times the number of solutions to the equation $g^t=1$, or the number of elements in $G$ of order dividing $t$. A more complicated system of equations is for example
$xy^2xy^3=1, yx^4yx^5=1$. The number of solutions to this equation will be equal to $n^{-2}$ times the basic invariant
$$\la^1(\l^1_1\l^2_1\l^2_2\l^1_2\l^2_3\l^2_4\l^2_5)\la^2(\l^2_6\l^1_3\l^1_4\l^1_5\l^1_6\l^2_7\l^1_7\l^1_8\l^1_9\l^1_{10}\l^1_{11})$$
where $\l^1$ and $\l^2$ are two copies of $\l$ and $\la^1$ and $\la^2$ are two copies of $\la$.
The number of solutions to some equation in a group is the same as the number of homomorphism from some finitely presented group $P$ to the group $G$ (where the generators of $P$ encode the indeterminates
and the relations between them encode the equations). We record this fact in the following lemma:
\begin{lemma}
All the basic invariants of $KG$ can be written as $n^{a}\# Hom_{Grp}(P,G)$ for some finitely presented group $P$ and some natural number $a$.
\end{lemma}
Since the basic invariants determine the isomorphism type of $H$, we have the following corollary:
\begin{corollary}
Let $G_1$ and $G_2$ be two finite groups. Then $G_1\cong G_2$ if and only if 
for every finitely presented group $P$ it holds that $\# Hom_{Grp}(P,G_1)=\# Hom_{Grp}(P,G_2)$.
\end{corollary}
As was pointed out in \cite{dks}, this corollary can be proved directly using the Inclusion-Exclusion Principle.
Indeed, if we know all the invariants of $KG$ we know in particular the order $n$ of $G$.
If $G_2$ is a group of order $n$, then $G_2\cong G$ if and only if there
exists an injective group homomorphism $G_2\ra G$.
The number of injective homomorphisms can be counted using the basic invariants
$\# Hom_{Grp}(G_2/N,G)$ for the different normal subgroups $N$ of $G_2$ and the Inclusion-Exclusion Principle.
\end{section}

\begin{section}{More examples of invariants}\label{ex2}
In this section we shall give intuitive interpretation of some of the basic invariants 	
when $H$ is a general finite dimensional semisimple Hopf algebra.
We begin with considering the element $\l_{1,2,\ldots,n}=\l_1\l_2\cdots \l_n$.
This element is central in $H$, and it can be written as $\sum_i \frac{dim(H)}{dim(W_i)}\nu_n(\psi_i)e_i$
where $\nu_n(\psi_i)$ is the $n$-th \textit{Frobenius Schur indicator} of $\psi_i$.
In \cite{KSZ} the following representation-theoretic interpretation was given to this scalar:
Consider the representation $W_i^{\ot n}$. The cyclic permutation $\sigma= (1,2,\ldots,n)$
of the tensor factors is not necessarily a homomorphism of representations.
It is true, however, that $$\sigma((W_i^{\ot n})^H) = (W_i^{\ot n})^H.$$
The proof of this follows from the fact that for any $H$-representation $V$, the map 
$V\mapsto V^H$ $v\mapsto \frac{1}{dim(H)}\l\cdot v$ is a projection onto the invariants,
and on the fact that $\l_1\ot \l_2\ot\cdots\ot \l_n$ is invariant under the cyclic permutation of the tensor factors.
We then have $$tr(\sigma|_{(W_i^{\ot n})^H}) = \frac{1}{dim(H)}\psi_i(\l_1\l_2\cdots\l_n)= \nu_n(\psi_i).$$
In particular, since $\nu_n(\psi_i)$ is the trace of an operator of order $n$, it lies in $\Z[\zeta_n]$
where $\zeta_n$ is a primitive $n$-th root of unity.
We then get the basic invariant $\la(\l_1\l_2\cdots\l_n) = dim(H)\sum_i dim(W_i)\nu_n(\psi_i)$
which is also contained in $\Z[\zeta_n]$.
In a similar way, we can study the trace of $\sigma^r$ and get representation-theoretic interpretations of other invariants.
For example, if $(r,n)=1$ we get an interpretation of $\la(\l_1\l_{r+1}\cdots \l_{1+r(n-1)})$ 
(where we take indices modulo $n$) as the sum of traces of operators of order $n$.
A more detailed study of these invariants can be found in the paper \cite{KSZ}.

The elements $\l_1\l_2\cdots \l_n$ and its counterpart in $H^*$, $\la_1\la_2\cdots \la_m$ 
can also be used to construct more complicated invariants.
To explain how, we begin with the following lemma:
\begin{lemma}
If $e^j$ is the central idempotent in $H^*$ which corresponds to the irreducible representation $V_j$,
then we have
$$e^j(e_i)= \frac{dim(W_i)dim(V_j)}{dim(H)}\psi_i(S(\phi_j))$$
\end{lemma}
\begin{proof}
This follows directly from the fact that $e_i=\frac{dim(W_i)}{dim(H)}\psi_i(S(\l_1))\l_2$.
A similar equation holds for $e^j$, and by pairing the two together we get the result.
\end{proof}
Using the last lemma, we can give an interpretation to more invariants.
Let $m$ and $n$ be two integers, and consider for example
$$\la(\l_1\l_{m+1}\cdots \l_{(n-1)m+1}\l_2\l_{m+2}\cdots\l_{(n-1)m+2}\cdots \l_m\l_{2m}\cdots \l_{nm}).$$
This invariant is equal to $(\la_1\la_2\cdots \la_m)(\l_1\l_2\cdots \l_n)$.
By using the fact that these elements are central in $H^*$ and in $H$ respectively, and by using the last lemma,
we get that this invariant is equal to 
$$dim(H)\sum_{i,j}\nu_n(\psi_i)\nu_m(\phi_j)\psi_i(S(\phi_j))$$
It is worth mentioning that for more complicated sequences there is no known representation theoretic
interpretation of the invariants.
I do not know, for example, if $\la(\l_1\l_3\l_2\l_4\l_5)$ can be written using the Frobenius-Schur indicators,
if it is necessarily contained in some cyclotomic extension of $\Q$ or not, or if it is an algebraic integer.

The fractions $\frac{dim(H)}{dim(W_i)}$ appear in a lot of the invariants. 
Kaplansky's Sixth Conjecture states that all these fractions are in fact integers.
It is true that if all the basic invariants of $H$ 
are algebraic integers then Kaplansky's Sixth Conjecture holds for $H$.
More precisely, consider the element 
$c_H^1 := \l^1_1\l^2_1S(\l^1_2)S(\l^2_2) = \sum_i \frac{dim(H)^2}{dim(W_i)^2}e_i$ in $Inv^{1,0}$ (this equality follows easily from Equation \ref{integral2}).
We claim the following proposition, which implies Theorem \ref{main6} immediately:
\begin{proposition}
If it holds that $\la((c_H^1)^n)\in \Z$ for every $n$, then $H$ satisfies Kaplansky's Sixth conjecture.
In particular, if all  the basic invariants of $H$ are algebraic integers, then $H$ satisfies Kaplansky's Sixth Conjecture.
\end{proposition}
\begin{proof}
We can think of multiplication by $c_H^1$ as a diagonal matrix $M$ in $M_d(\Q)$ where $d=dim(H)$.
All the eigenvalues of $c_H^1$ are $\frac{dim(H)^2}{dim(W_i)^2}$. 
So it will be enough to prove that if $tr(M^n)\in \Z$ for every $n$ then all the eigenvalues of $M$ are integral over $\Z$.

We can prove this by localizing at the different primes of $\Z$. 
Let $p$ be a prime number, and assume that $m$ is the smallest natural number such that all the eigenvalues
of $p^mM$ are contained in $\Z_{(p)}=\{\frac{a}{b}|p\nmid b\}$.
Assume that $m>0$. Let $N$ be an integer such that $p^{N-1}>d$.
We write $M'=(p^mM)\, mod\, p^N$. Then for a large enough $r$ we will have that $M''=M'^{r(p^{N-1}(p-1))}$
is a diagonal matrix which contains only the eigenvalues 0 and 1 and that $p^{mrp^{N-1}(p-1)}|tr(M'')$.
Since $0\leq tr(M'')\leq d< p^{N-1}$ and $m>0$ we get that $tr(M'')=0$ (we consider $M''$ as a matrix over $\Z/p^N$).
By the assumption on $N$, we have that this is possible if and only if $M''$ is the zero matrix.
But this contradicts the minimality of $m$, and therefore $m=0$.
This implies that all the eigenvalues are contained in $\bigcap_p \Z_{(p)} = \Z$ as desired.

In order to prove the second part of the proposition, we just need to show that all the scalars $\la((c_H^1)^n)$ are basic invariants.
We write $$\la((c_H^1)^n)  = \la_1(\l^1_1\l^2_1S(\l^1_2)S(\l^2_2))\cdots\la_n(\l^{2n-1}_1\l^{2n}_1S(\l^{2n-1}_2)S(\l^{2n}_2)) = $$
$$ \la_1(\l^1_1)\la_2(\l^2_2)\la_3(S(\l^1_2))\la_4(S(\l^2_2))\cdots $$ $$\la_{4n-3}(\l^{2n-1}_1)\la_{4n-2}(\l^{2n}_2)\la_{4n-1}(S(\l^{2n-1}_2))\la_{4n}(S(\l^{2n}_2)).$$
We can now write the antipode $S(\l^i_j)$ as $\l^i_j\cdots l^i_{j+m-2}$, by Equation \ref{antipode}.
By using again the fact that the multiplication in $H$ is dual to the comultiplication in $H^*$, we get a representation of $\la((c_H^1)^n)$ as a basic invariant, as desired.
\end{proof}
The last example we give here of basic invariants which has a representation theoretic interpretation is due to Shimizu.
We begin by recalling that if $H$ is any finite dimensional Hopf algebra, then the Drinfeld double $D(H)$
is another finite dimensional Hopf algebra of dimension $dim(H)^2$.
This Hopf algebra is quasi-triangular: if $V$ and $W$ are representations of $D(H)$,
then we have a natural isomorphism $c_{V,W}:V\ot W\ra W\ot V$ of $D(H)$-representations.
Moreover, the family $\{c_{V,W}\}$ of isomorphism will satisfy certain braid relations.
In particular, the representation $V^{\ot n}$ is in a natural way a representation of the Braid Group on $n$ strings $B_n$.

The vector spaces $H$, $H^*$ and $D(H)$ carry a natural $D(H)$-module structure.
So if we take $g\in B_n$, and $V$ to be one of $H$, $H^*$ or $D(H)$, 
we get the scalar $tr(g_{V^{\ot n}})$. A direct calculation using the $D(H)$
action and the braid group action on these spaces reveals the fact that these scalars are also basic invariants.
Etingof, Rowell and Witherspoon proved in \cite{ERW} that if $H$ is group theoretical then the action of the braid group 
always factors over some finite quotient of the braid group (their result is more general, and holds for braided group theoretical categories).
This implies that in this case these basic invariants will also be contained in $\Z[\zeta_m]$ for some $m$.
Naidu and Rowell conjectured in \cite{NR} that this action factors over a finite quotient for all Hopf algebras (and in fact, for all braided weakly integral fusion categories)
and gave some more examples in which it holds. 

The result of Shimizu concerns the representation $V=D(H)$.
To state the result, we need to recall a few facts about three manifolds (for more details see the paper \cite{Shimizu}).
Let $g\in B_n$ be a braid.
By ``closing up'' $g$, we get a link. By embedding this link in $S^3$ and preforming Dehn Surgery,
we get a 3-manifold $M_g$. 
The Reshetikhin-Turaev invariant $RT_{D(H)}(M_g)$ gives a scalar invariant of $M_g$, depending on the quasi-triangular Hopf algebra $D(H)$
(actually, it depends only on the braided representation category $Rep-D(H)$).
Shimizu proved in \cite{Shimizu} that $RT_{D(H)}(M_g) = tr_{D(H)^{\ot n}}(g)$.
In other words, The Reshetikhin-Turaev invariants can be thought of as invariants of 3-manifolds parametrized by semisimple Hopf algebras.
From the Hopf-algebraic point of view, we can also think of them as invariants of Hopf algebras parametrized by 3-manifold 
(since every orientable compact 3-manifold is of the form $M_g$ for some $g$, and if $M_g$ is homeomorphic with $M_{g'}$,
then the resulting invariants for $H$ will be the same).
In case $H=KG$, Shimizu proved that the finitely presented group $P$ we get in the expression of the invariant is
the fundamental group $P=\pi_1(M_g)$ of $M_g$.
\end{section}
\begin{section}*{Acknowledgements}
I first encountered Geometric Invariant Theory during a program on moduli spaces at the Isaac Newton Institute in Cambridge
at the first half of 2011.
I would like to thank the Newton Institute and the organizers of the program. 
During the writing of this paper I was supported by the Danish National Research Foundation (DNRF) through the Center for Symmetry and Deformation.
\end{section}

\end{document}